\documentclass[12pt]{article}
\usepackage{graphicx}
\usepackage{amsmath}
\usepackage{amsfonts}
\usepackage{amsthm}
\usepackage[T1]{fontenc}
\usepackage{url}
\usepackage{color}
\usepackage[margin=1in]{geometry}
\usepackage{amssymb,bm}
\usepackage{amsmath}
\usepackage{amsthm}
\usepackage{graphicx}
\usepackage[active]{srcltx} 
\usepackage{hyperref}
\hypersetup{pdfborder=0 0 0}

\newtheorem{theorem}{{\sc Theorem}}[section]

\newtheorem{lemma}[theorem]{{\sc Lemma}}

\newtheorem{remark}[theorem]{Remark}

\newtheorem{definition}[theorem]{Definition}

\newcommand{\dOm}{\partial\Omega}

\def\XXint#1#2#3{{\setbox0=\hbox{$#1{#2#3}{\int}$ }
\vcenter{\hbox{$#2#3$ }}\kern-.6\wd0}}

\newcommand{\Ga}{\alpha}

\newcommand{\Gd}{\delta}
\newcommand{\Ge}{\epsilon}
\newcommand{\eps}{\varepsilon}
\newcommand{\Gve}{\varepsilon}
\newcommand{\Gf}{\phi}

\newcommand{\Gg}{\gamma}

\newcommand{\Gl}{\lambda}

\newcommand{\Gth}{\theta}

\newcommand{\GG}{\Gamma}

\newcommand{\GL}{\Lambda}

\newcommand{\GO}{\Omega}

\bmdefine\BGa{\alpha}
\bmdefine\BGb{\beta}
\bmdefine\BGd{\delta}
\bmdefine\BGe{\epsilon}
\bmdefine\BGve{\varepsilon}
\bmdefine\BGf{\phi}
\bmdefine\BGvf{\varphi}
\bmdefine\BGg{\gamma}
\bmdefine\BGc{\chi}
\bmdefine\BGi{\iota}
\bmdefine\BGk{\kappa}
\bmdefine\BGl{\lambda}
\bmdefine\BGn{\eta}
\bmdefine\BGm{\mu}
\bmdefine\BGv{\nu}
\bmdefine\BGp{\pi}
\bmdefine\BGth{\theta}
\bmdefine\BGvth{\vartheta}
\bmdefine\BGr{\rho}
\bmdefine\BGvr{\varrho}
\bmdefine\BGs{\sigma}
\bmdefine\BGvs{\varsigma}
\bmdefine\BGt{\tau}
\bmdefine\BGj{\tau}
\bmdefine\BGu{\upsilon}
\bmdefine\BGo{\omega}
\bmdefine\BGx{\xi}
\bmdefine\BGy{\psi}
\bmdefine\BGz{\zeta}
\bmdefine\BGD{\Delta}
\bmdefine\BGF{\Phi}
\bmdefine\BGG{\Gamma}
\bmdefine\BGL{\Lambda}
\bmdefine\BGP{\Pi}
\bmdefine\BGT{\Theta}
\bmdefine\BGS{\Sigma}
\bmdefine\BGU{\Upsilon}
\bmdefine\BGO{\Omega}
\bmdefine\BGX{\Xi}
\bmdefine\BGY{\Psi}

\bmdefine\BCA{{\mathcal A}}
\bmdefine\BCB{{\mathcal B}}
\bmdefine\BCC{{\mathcal C}}
\bmdefine\BCD{{\mathcal D}}
\bmdefine\BCE{{\mathcal E}}
\bmdefine\BCF{{\mathcal F}}
\bmdefine\BCG{{\mathcal G}}
\bmdefine\BCH{{\mathcal H}}
\bmdefine\BCI{{\mathcal I}}
\bmdefine\BCJ{{\mathcal J}}
\bmdefine\BCK{{\mathcal K}}
\bmdefine\BCL{{\mathcal L}}
\bmdefine\BCM{{\mathcal M}}
\bmdefine\BCN{{\mathcal N}}
\bmdefine\BCO{{\mathcal O}}
\bmdefine\BCP{{\mathcal P}}
\bmdefine\BCQ{{\mathcal Q}}
\bmdefine\BCR{{\mathcal R}}
\bmdefine\BCS{{\mathcal S}}
\bmdefine\BCT{{\mathcal T}}
\bmdefine\BCU{{\mathcal U}}
\bmdefine\BCV{{\mathcal V}}
\bmdefine\BCW{{\mathcal W}}
\bmdefine\BCX{{\mathcal X}}
\bmdefine\BCY{{\mathcal Y}}
\bmdefine\BCZ{{\mathcal Z}}

\bmdefine\Bzr{ 0}
\bmdefine\Ba{ a}
\bmdefine\Bb{ b}
\bmdefine\Bc{ c}
\bmdefine\Bd{ d}
\bmdefine\Be{ e}
\bmdefine\Bf{ f}
\bmdefine\Bg{ g}
\bmdefine\Bh{ h}
\bmdefine\Bi{ i}
\bmdefine\Bj{ j}
\bmdefine\Bk{ k}
\bmdefine\Bl{ l}
\bmdefine\Bm{ m}
\bmdefine\Bn{ n}
\bmdefine\Bo{ o}
\bmdefine\Bp{ p}
\bmdefine\Bq{ q}
\bmdefine\Br{ r}
\bmdefine\Bs{ s}
\bmdefine\Bt{ t}
\bmdefine\Bu{ u}
\bmdefine\Bv{ v}
\bmdefine\Bw{ w}
\bmdefine\Bx{ x}
\bmdefine\By{ y}
\bmdefine\Bz{ z}
\bmdefine\BA{ A}
\bmdefine\BB{ B}
\bmdefine\BC{ C}
\bmdefine\BD{ D}
\bmdefine\BE{ E}
\bmdefine\BF{ F}
\bmdefine\BG{ G}
\bmdefine\BH{ H}
\bmdefine\BI{ I}
\bmdefine\BJ{ J}
\bmdefine\BK{ K}
\bmdefine\BL{ L}
\bmdefine\BM{ M}
\bmdefine\BN{ N}
\bmdefine\BO{ O}
\bmdefine\BP{ P}
\bmdefine\BQ{ Q}
\bmdefine\BR{ R}
\bmdefine\BS{ S}
\bmdefine\BT{ T}
\bmdefine\BU{ U}
\bmdefine\BV{ V}
\bmdefine\BW{ W}
\bmdefine\BX{ X}
\bmdefine\BY{ Y}
\bmdefine\BZ{ Z}

\DeclareMathOperator{\sgn}{sgn}
\newcommand{\dx}{\,dx}
\newcommand{\Lp}[1]{L^{p}(#1)}
\newcommand{\norm}[1]{\left\lVert #1 \right\rVert}
\newcommand{\dist}{\operatorname{dist}}
\newcommand{\diam}{\operatorname{diam}}

\newcommand{\SO}{\operatorname{SO}}

\begin{document}

\title{On weighted rigidity estimates for bulk and thin domains}

\author{Davit Harutyunyan\thanks{University of California Santa Barbara, harutyunyan@math.ucsb.edu}
and Andre Martins Rodrigues\thanks{University of California Santa Barbara, andre02@ucsb.edu}}
\maketitle

\begin{abstract}

This work is concerned with weighted Geometric Rigidity Estimates and Korn's first inequalities in bulk and thin domains. 
We consider weights, that are a nonnegative power of the distance function to part of the boundary of the domain. For the case of thin domains, the distance is taken to be from the thin face of the domain boundary, and the constants depend on the domain thickness and when the mid-surface contains a flat region, the constants are proven to have optimal scaling as the thickness approaches zero. We employed some covering techniques utilized by Acosta, Cejas, and Duran in [\ref{bib:Aco.Cej.Dur.}] to prove a weighted Poincar\'e inequality, and later by Conti and Zwicknagl in [\ref{bib:Con.Zwi.}] to prove weighted Poincar\'e and classical Geometric Rigidity inequalities in Lipschitz domains. However, because the distance is taken to be only from part of the boundary, the covering parts become more delicate, especially for thin domains and the analysis becomes non-straightforward. 

\end{abstract}

\section{Introduction}
\label{Sec:1}

This paper explores important mathematical tools used in the theories of linear and nonlinear elasticity. It is within this context that Korn's and Geometric Rigidity type inequalities emerge as pivotal analytical tools. Korn's inequalities, originating from contributions made in by Korn [\ref{bib:Korn.1},\ref{bib:Korn.2}], have played a crucial role in analyzing boundary value problems in the theory of linear elasticity [\ref{bib:Horgan},\ref{bib:Friedrichs},\ref{bib:Kon.Ole.1},\ref{bib:Kon.Ole.2},\ref{bib:Ciarlet},\ref{bib:Aco.Dur.Gar.},\ref{bib:Koh.Vog.},\ref{bib:Mueller}]. They are crucial in establishing the existence of energy minimizers [\ref{bib:Horgan},\ref{bib:Friedrichs},\ref{bib:Kon.Ole.1},\ref{bib:Kon.Ole.2},\ref{bib:Kohn}], and have been central in both linear and nonlinear shell theories. In the book [\ref{bib:Tov.Smi.}], Tovstik and Smirnov discuss how the asymptotics of Korn's constant in Korn's first inequality is linked to the shell deformation in various situations. 

It has been shown by Friesecke, James and M\'ueller in [\ref{bib:Fri.Jam.Mue.1},\ref{bib:Fri.Jam.Mue.2}], that reduced shell theories derived from three dimensional nonlinear elasticity depend on the scaling of the constant in the domain thickness in the Geometric Rigidity Estimate, see also [\ref{bib:Tov.Smi.},\ref{bib:Mueller},\ref{bib:Lew.Mue.},\ref{bib:Harutyunyan.1},\ref{bib:Harutyunyan.2},\ref{bib:Harutyunyan.3},\ref{bib:Gra.Har.1},\ref{bib:Gra.Har.2},\ref{bib:Gra.Tru.},\ref{bib:Yao}]. In their rigorous theory of buckling of slender structures, Grabovsky and Truskinovsky and later Grabovsky and the first author proved, that there is a relation between Korn's constant in Korn's first inequality and the critical buckling load of a given slender structure [\ref{bib:Gra.Tru.},\ref{bib:Gra.Har.1}]. Korn's first inequality in SBD has been proven in [\ref{bib:Cha.Con.Fra.},\ref{bib:Friedrich.1},\ref{bib:Friedrich.2},\ref{bib:Cag.Cha.Sca.}], which proved to be crucial in the theory of fracture  [\ref{bib:Cha.Con.Fra.},\ref{bib:Cha.Con.Iur.},\ref{bib:Con.Foc.Iur.}]. A Korn's inequality with non-constant coefficients have been proven in [\ref{bib:Neff}], for incompatible fields in 
 [\ref{bib:Lew.Mue.Nef.},\ref{bib:Gme.Lew.Nef.}]. 

The focus of this work is weighted Geometric Rigidity and Korn's first inequalities without boundary conditions for bulk and thin domains. These adaptations are important for scenarios necessitating the use of polar coordinates or distinguishing between longitudinal and transverse directions, such as in studies involving junctions of massive bodies and thin rods [\ref{bib:Harutyunyan.1},\ref{bib:Har.Mik.}].
Firstly, they are indispensable when standard Cartesian coordinates are not suitable in the given situation, and a shift to polar coordinates or other variable transformations is required, see [\ref{bib:Harutyunyan.1}]. Additionally, the use of weights, particularly those from part of boundary points, are useful specifically in deriving asymptotically sharp forms rigidity inequalities for junctions of thin plates and the study of junctions between massive bodies and thin rods. 

Most of the results were inspired by the papers Acosta, Cejas, and Duran in [\ref{bib:Aco.Cej.Dur.}] and Conti and Zwicknagl in [\ref{bib:Con.Zwi.}]. In [\ref{bib:Con.Zwi.}], employing an well-established covering technique (e.g. [\ref{bib:Aco.Cej.Dur.}]), Conti and Zwicknagl provide a new proof of the Geometric Rigidity Estimate with explicit constants. The point is, it was generally accepted in the community given the existing proofs and techniques, that the Geometric Rigidity Estimate holds in Lipschitz domains with the constant depending only on the domain Lipschitz character and the domain diameter. A complete proof of this fact was presented in [\ref{bib:Con.Zwi.}] with explicit specifications, which was necessary for the aim of their paper. It is worth mentioning, that the Poincar\'e inequality proven in [\ref{bib:Con.Zwi.}], was already proven in [\ref{bib:Hurri},\ref{bib:Dre.Dur.},\ref{bib:Aco.Cej.Dur.}]. Versions of weighted Korn inequalities for bulk domain have been proven in [\ref{bib:Aco.Dur.Lam.}], and for shells have been proven in [\ref{bib:Harutyunyan.1},\ref{bib:Har.Mik.}]. 

In this paper, we will use similar techniques to prove another version of a weighted Poincaré inequality, that will be utilized to prove weighted versions of the Geometric Rigidity and Korn's first inequalities in bulk domains. Additionally, we prove analogues sharp weighted inequalities for thin domains. As the distance is taken to be only from part of the boundary, the standard Whitney cover does not work anymore and the Whitney-like covering parts of the proofs become somewhat more delicate, especially for thin domains, where one has to construct an $h-$scale uniform cover. Each of the rigidity estimate requires its own nontrivial cover, which we will construct in a separate lemma within the proof of the theorem. We have decided to provide detailed proofs with explicit constants like in [\ref{bib:Con.Zwi.}] in most cases where things are not obvious, for maximal rigor. 

Most of the results in the present paper were part of the PhD thesis of A. M. Rodrigues, and were reported in his dissertation [\ref{bib:Rodrigues}] in October 2023 (published online in March, 2024).  The theorems for thin domains were reported for plates only in [\ref{bib:Rodrigues}], and the paper extends them for $C^2$ thin domains. Some of the theorems are new.


\section{Definitions and main results}
\setcounter{equation}{0}
\label{sec:2}

For the weighted Poincaré and Korn inequalities to hold, one needs to assume some conditions on the domain. In this paper we will assume that the domain is open, connected, and uniformly Lipschitz.  More, precisely, we will adopt the framework in [\ref{bib:Con.Zwi.}] in terms of the definition of uniformly Lipschitz domains to provide constants with explicit dependence on the domain parameters. 

\begin{definition}[Uniformly Lipschitz domains] 
\label{Def:2.1} Let $L, R>0$. An open set $\Omega \subseteq \mathbb{R}^n$ is $(L, R)$-Lipschitz if there is $\Gve>0$ such that:
\begin{enumerate}
    \item $\mathrm{diam}(\Omega)<R\varepsilon.$ 
    \item For every $x \in \partial \Omega$ there are $f_x \in \operatorname{Lip}\left(\mathbb{R}^{n-1} ; \mathbb{R}\right)$ with $\operatorname{Lip}\left(f_x\right) \leq L$ and an isometry $\BA_x: \mathbb{R}^n \rightarrow \mathbb{R}^n$ such that $B_\varepsilon(x) \cap \Omega=B_\varepsilon(x) \cap V_x$, where
        $$
            V_x:=\BA_x\left\{\left(y^{\prime}, y_n\right) \in \mathbb{R}^{n-1} \times \mathbb{R}: y_n<f_x\left(y^{\prime}\right)\right\}
        $$
\end{enumerate}
\end{definition}

So, for uniform Lipschitz domains, the size of the balls covering the boundary is comparable with the diameter of the domain. We will show that weighted Korn's first and Geometric Rigidity Estimates hold in such domains that are also connected. The sets of proper rotations and skew-symmetric matrices on $\mathbb R^n$ are denoted by $\operatorname{SO}(n)$ and 
$\mathbb{R}_{\mathrm{skew}}^{n \times n}$ respectively. Also, we will use boldface capital letters for matrices.


\begin{theorem}[Weighted rigidity estimates on bulk domains]
\label{Thm:2.2} 

Let $\Omega \subset \mathbb{R}^n$ be an open, connected $(L, R)$-Lipschitz set, let $\GG$ be a nonempty closed subset of 
$\partial\Omega,$ and denote $\delta_\GG(x)=\mathrm{dist}(x, \GG)$. Assume $p \in(1, \infty)$ and $\alpha\geq 0.$ Then there exists a constant $C=C(n,p,\alpha,L,R)$, such that for any $u \in W_{\mathrm{loc}}^{1, p}\left(\Omega ; \mathbb{R}^n\right)$, there is a skew-symmetric matrix $\BA \in \mathbb{R}_{\mathrm{skew}}^{n \times n},$ and a proper rotation $\BR \in \operatorname{SO}(n),$ such that
\begin{equation}
\label{2.1}
\|(\delta_\GG)^\alpha(\nabla u-\BA)\|_{L^p(\Omega)} \leq C \left\|(\delta_\GG)^\alpha e(u)\right\|_{L^p(\Omega)}.
\end{equation}
and
\begin{equation}
\label{2.2}
\|(\delta_\GG)^\alpha(\nabla u-\BR)\|_{L^p(\Omega)} \leq C\|(\delta_\GG)^\alpha\operatorname{dist}(\nabla u, \operatorname{SO}(n))\|_{L^p(\Omega)}.
\end{equation}
\end{theorem}

\vspace{0.4cm}

An analogous result with sharp constants holds for thin domains as well. Before we formulate the theorem, let us recall what a thin domain is. Adopting the terminology in [\ref{bib:Lew.Mue.}] (see also [\ref{bib:Harutyunyan.3}]), assume $S\subset\mathbb R^3$ is a connected compact regular $C^1$ globally bi-Lipschitz surface with a unit normal field $n(x)\colon S\to \mathbb S^{2}. $ Let $h>0$ be a small parameter and assume the family of Lipschitz functions $g_1^h(x),g_2^h(x)\colon S\to (0,\infty)$ satisfies the uniform conditions
\begin{equation}
\label{2.3}
h\leq g_1^h(x),g_2^h(x)\leq c_1 h\quad \text{and}\quad |\nabla g_1^h(x)|+|\nabla g_2^h(x)|\leq c_2h
\quad\text{for all}\quad x\in S.
\end{equation}
Then the set 
\begin{equation}
\label{2.3'}
\Omega_h=\{x+tn(x) \ : \ x\in S,\ t\in [-g_1^h(x),g_2^h(x)]\}
\end{equation}
is a shell with thickness of order $h,$ which is also called a thin domain of thickness of order $h.$ The surface $S$ is called the mid-surface of the thin domain $\Omega_h.$

The next theorem is the analogue of Theorem~\ref{Thm:2.2} for thin domains, with the distance taken to be only from the thin face of the domain boundary. Due to the specifics of that distance, one needs to assume at least some kind of Lipschitz regularity on the boundary of the mid-surface $S.$ Below is hypothesis (H) that will be assumed in Theorems~\ref{Thm:2.3} and \ref{Thm:2.4}. concerning thin domains.\\ 
\noindent\textbf{(Hypothesis H).}  
\begin{itemize}
\item[(H1)] Assume $S\subset\mathbb R^3$ is a connected compact regular $C^2$ globally bi-Lipschitz surface with a unit normal field $n(x)\in C^1(S;\mathbb S^{2})$ and with a nonempty boundary $\partial S.$

\item[(H2)] There exist a radius $r_S>0$ and a constant $L_S>0,$ such that for every $x\in S$, in the frame $(T_xS,n(x))$ centered at $x,$ one has $S\cap B(x,r_S)$ is the graph of a $C^{2}$ function $\Gf_x$ over an open subset
$0\in \mathcal O_x\subset T_xS$ with $\Gf_x(0)=0$, $\nabla \Gf_x(0)=0$ and $|\nabla \Gf_x|\le\tfrac12$. Here, $T_xS$ is the tangent space to $S$ at $x.$\footnote{ Note, that this condition automatically follows from the compactness and $C^2$-regularity of $S$ i.e., (H1), but we formulate is within the hypothesis for convenience in applications, as (H3) contains the same constants and is not a direct consequnce of the compactness of and $C^2$-regularity of $S.$}
\item[(H3)] If $\mathrm{dist}(x,\partial S)\ge r_S$ then $B'_{r_S/2}\subset \mathcal O_x$ in the frame $(T_xS,n(x))$. If $x\in\partial S$ then there exists a Lipschitz function $\psi_x\colon\mathbb R\to\mathbb R$ such that
$$
\mathcal O_x\cap B'_{r_S}=\{y=(y_1,y_2)\in B'_{r_S}:\ y_2<\psi_x(y_1)\},\qquad
\mathrm{Lip}(\psi_x)\le L_S,\ \psi_x(0)=0,
$$
and $\partial S\cap B(x,r_S)$ corresponds to the graph of $\psi_x$ within $B(x,r_S).$
\end{itemize}
In what follows, hypothesis (H) means all (H1)-(H3) together.


\begin{theorem}[Sharp weighted rigidity estimates for thin domains]
\label{Thm:2.3} 
Let $S\subset\mathbb R^3$ satisfy hypothesis (H) and let $g_1^h(x),g_2^h(x)\colon S\to (0,\infty)$ and $\Omega_h$ be as in (\ref{2.3}) and (\ref{2.3'}). Denote the lateral boundary of $\Omega_h$ by $\partial_S\Omega_h=\{x+tn(x) \ : \  x\in\partial S, \ \ t\in  [-g_1^h(x),g_2^h(x)]\},$ and denote the distance from the lateral boundary $\delta'(x)=\mathrm{dist}(x,\partial_S\Omega_h)$ for $x\in\mathbb R^3.$ Assume further $p \in(1, \infty)$ and $\alpha\geq 0$. There exist constants $h_0=h_0(S,c_1)>0$ and 
$C=C(S,c_1,c_2,p,\alpha)>0,$ such that for any vector field $u \in W^{1, p}\left(\Omega_h ; \mathbb{R}^3\right),$ there exists a proper rotation $\BR \in \operatorname{SO}(3)$ and a skew-symmetric matrix $\BA \in \mathbb{R}_{\mathrm{skew}}^{3 \times 3},$ such that for any $h\in (0,h_0)$ one has
 \begin{equation}
\label{2.4}
\|(\delta')^\alpha(\nabla u-\BA)\|_{L^p(\Omega_h)} \leq  \frac{C}{h}\left\|(\delta')^\alpha e(u)\right\|_{L^p(\Omega_h)},
\end{equation}
and
\begin{equation}
\label{2.5}
\|(\delta')^\alpha(\nabla u-\BR)\|_{L^p(\Omega_h)} \leq \frac{C}{h}\|(\delta')^\alpha\operatorname{dist}(\nabla u, \operatorname{SO}(3))\|_{L^p(\Omega_h)}.
\end{equation}
Moreover, if the mid-surface $S$ has a flat region, i.e., a relatively open subset of $S$ is contained in a two dimensional hyperplane, then the asymptotics of the constants in both inequalities is sharp. 
\end{theorem}

Like in [\ref{bib:Con.Zwi.}], Theorems~\ref{Thm:2.2} and \ref{Thm:2.3} will be proven by means of a weighted Poincaré inequality, which we formulate below. We will need to prove the inequality both for bulk and thin domains for applications. While the weighted version for bulk domains is a straightforward extension of the already standard weighted Poincae\'e inequality, the one for thin domains is new.


\begin{theorem}[Weighted Poincaré Inequality] 
\label{Thm:2.4} 

\begin{itemize}
\item[(i)] Let $\Omega \subset \mathbb{R}^n$ be open, connected $(L, R)$-Lipschitz set, and let $\GG$ be a nonempty closed subset of $\dOm$ and let $\delta_\GG(x)=\mathrm{dist}(x, \GG)$. Assume $p \in[1, \infty)$ 
and $\alpha\geq 0.$ There exists a constant $C=C(n,p,\alpha,L,R)$, such that for any $u \in W_{\mathrm{loc}}^{1, p}\left(\Omega ; \mathbb{R}^k\right)$, there exists $a \in \mathbb{R}^k$ such that
\begin{equation}
\label{2.6}
\|(\delta_\GG)^\alpha (u-a)\|_{L^p(\Omega)} \leq C \|(\Gd_\GG)^{1+\Ga} \nabla u\|_{L^p(\Omega)} .
\end{equation}

\item[(ii)] Let $S, \Omega_h,\delta',p$ and $\alpha$ be as in Theorem~\ref{Thm:2.3}. There exists constants\\ 
$C=C(S,c_1,c_2,p,\alpha)>0$ and $h_0=h_0(S,c_1)>0$ such that for any $h\in (0,h_0)$ and any 
$u \in W^{1, p}\left(\Omega_h ; \mathbb{R}^3\right)$, there exists $a \in \mathbb{R}^3$ such that
\begin{equation}
\label{2.7}
\|(\delta')^\alpha (u-a)\|_{L^p(\Omega_h)} \leq C \|(\Gd')^{1+\Ga} \nabla u\|_{L^p(\Omega_h)} .
\end{equation}

\end{itemize}
\end{theorem}


\section{Auxiliary lemmas}
\setcounter{equation}{0}
\label{Sec:3}

The first step is to extend [Lemma 5.6, \ref{bib:Con.Zwi.}], which can be considered a Weighted Poincar\'e inequality in 1D.


\begin{lemma}
\label{Lem:3.1}
Let $I=(a,b)\subset\mathbb{R}$ and $E=[a,a+\epsilon]$ with $0<\epsilon<b-a$. Write
$\beta:=\phi(a)$. Then for all $\phi\in C^{1}(\overline I)$, $\alpha\ge 0$, $\lambda\in\mathbb{R}$
and $p\in[1,\infty)$ one has

\begin{align}
\label{3.1}
\int_{I} \bigl |(b-x)^{\alpha}&(\phi-\lambda)\bigr |^{p}dx\le \\ \nonumber
&2^{p-1}\left(\frac{|I|}{|E|}\right)^{\alpha p+1}\int_{E}\bigl|(b-x)^{\alpha}(\beta-\lambda)\bigr|^{p}dx+
2^{p-1}\left(\frac{p}{\alpha p+1}\right)^{\!p}\!\!\int_{I}\bigl|(b-x)^{1+\alpha}\phi'\bigr|^{p}dx .
\label{3.1}
\end{align}

\end{lemma}


\begin{proof}

\medskip
\noindent\textbf{Step 1.} We first prove the weighted Hardy inequality
\begin{equation}
\label{3.2}
\norm{(b-x)^{\alpha}(\phi-\beta)}_{\Lp{I}}  \le \frac{p}{\alpha p+1}\norm{(b-x)^{1+\alpha}\phi'}_{\Lp{I}} .
\end{equation}
Set $g=\phi-\beta$, so that $g(a)=0$, and define
$$
W(x)=\int_{x}^{b}(b-s)^{\alpha p}ds=\frac{(b-x)^{\alpha p+1}}{\alpha p+1},
\quad\text{so}\quad 
W'=-(b-x)^{\alpha p},\quad W(b)=0 .
$$
Hence we get integrating by parts,
\begin{align*}
\int_{I}(b-x)^{\alpha p}|g|^{p}\dx&=\int_{I}(-W')\,|g|^{p}\dx\\
&=\Bigl[-W|g|^{p}\Bigr]_{a}^{b}+\int_{I}W\,\frac{d}{dx}|g|^{p}\dx .
\end{align*}
The boundary term vanishes,  thus using $\dfrac{d}{dx}|g|^{p}=p\,|g|^{p-1}\sgn(g)\,\phi'$ together with $W\ge 0$ we obtain
$$
\int_{I}(b-x)^{\alpha p}|g|^{p}\dx \le \frac{p}{\alpha p+1}\int_{I}(b-x)^{\alpha p+1}\,|g|^{p-1}|\phi'|\dx .
$$
For $p=1$ we are done. Assume now $p>1.$ Next split the weight as $(b-x)^{\alpha p+1}=(b-x)^{\alpha(p-1)}\,(b-x)^{1+\alpha}$ (the exponents add to $\alpha p+1$), so that the integrand equals
$\bigl[(b-x)^{\alpha}|g|\bigr]^{p-1}\cdot(b-x)^{1+\alpha}|\phi'|$, and apply H\"older inequality with
exponents $p'$ and $p$ to get
$$
\int_{I}(b-x)^{\alpha p}|g|^{p}\dx \le  \frac{p}{\alpha p+1} \norm{(b-x)^{\alpha}g}_{\Lp{I}}^{p-1}
\norm{(b-x)^{1+\alpha}\phi'}_{\Lp{I}} .
$$
Dividing by $\norm{(b-x)^{\alpha}g}_{\Lp{I}}^{p-1}$ yields (\ref{3.2}) (if this norm
vanishes then $g\equiv 0$ and (\ref{3.2}) is trivial). \\
\medskip
\noindent\textbf{Step 2.} Now we prove an estimate on $E$. For the constant $c=\beta-\lambda$, the substitution $t=b-x$ gives
$$
\frac{\displaystyle\int_{I}(b-x)^{\alpha p}\dx}{\displaystyle\int_{E}(b-x)^{\alpha p}\dx}
=\frac{|I|^{\alpha p+1}}{|I|^{\alpha p+1}-(|I|-|E|)^{\alpha p+1}}\\
=\frac{1}{1-(1-\theta)^{q}},
$$
where
$$
\theta=\frac{|E|}{|I|}\in[0,1],\quad q=\alpha p+1\ge 1 .
$$
Since $q\ge 1$ implies $\theta^{q}+(1-\theta)^{q}\le\theta+(1-\theta)=1$, we have
$1-(1-\theta)^{q}\ge\theta^{q}$, hence
\begin{equation}
\label{3.3}
\norm{(b-x)^{\alpha}(\beta-\lambda)}_{\Lp{I}}^{p}
\;\le\;
\left(\frac{|I|}{|E|}\right)^{\!\alpha p+1}
\norm{(b-x)^{\alpha}(\beta-\lambda)}_{\Lp{E}}^{p} .
\end{equation}

\medskip
\noindent\textbf{Step 3.} In the last step we combine the obtained estimates. By the triangle inequality in $\Lp{I}$ and the convexity inequality
$(u+v)^{p}\le 2^{p-1}(u^{p}+v^{p}),$ for $u,v\geq 0,$ we arrive at
$$
\norm{(b-x)^{\alpha}(\phi-\lambda)}_{\Lp{I}}^{p}
\le
2^{\,p-1}\norm{(b-x)^{\alpha}(\phi-\beta)}_{\Lp{I}}^{p}
+2^{\,p-1}\norm{(b-x)^{\alpha}(\beta-\lambda)}_{\Lp{I}}^{p} .
$$
Inserting (\ref{3.2}) into the first term and (\ref{3.3}) into the second gives (\ref{3.1}).
\end{proof}

The next step is to use the properties of uniformly Lipschitz domains and the previous result to prove similar estimates in higher dimensions close to the boundary of the domain. This is typically done by localization of the boundary using a cover and applying 
one-dimensional estimates locally. The number $\varepsilon /(5+4L)$ will appear in several places of the manuscript, so we set for brevity  $\varepsilon_L=\varepsilon /(5+4L)$ in what follows.


\begin{lemma}
\label{Lem:3.2}
     Let $\Omega \subset \mathbb{R}^n,$ $L,R, \varepsilon,p,\alpha,\GG,\Gd_\GG$ be as in Theorem~\ref{Thm:2.2} and let 
     $x_* \in \GG\subset\dOm,\ r \in(0, \varepsilon_L].$ There exists a constant $C=C(n,p,\Ga,L),$ such that for any $u \in W_{\mathrm{loc}}^{1, p}(\Omega,\mathbb R)$ there exist 
 $a \in \mathbb{R}$ such that
\begin{equation}
\label{3.4}
\|\delta_\GG(x)^\Ga (u-a)\|_{L^p(\Omega\cap B_r(x_*))} \leq C \|\delta_\GG(x)^{(\alpha+1)}\nabla u\|_{L^p(\Omega)}.
\end{equation}
\end{lemma}


\begin{proof}
    By the definition of $(L,R)$-Lipschitz domains, we have that 
    $B_\varepsilon(x_*) \cap \Omega=B_\varepsilon (x_*) \cap V$, where 
$$
V=A\left\{\left(y^{\prime}, y_n\right) \in \mathbb{R}^{n-1} \times \mathbb{R}: y_n<f\left(y^{\prime}\right)\right\},
$$ 
after an isometry $\BA$ and for an $L$-Lipschitz function $f$. Since all the inequalities are invariant under rotations and translations we can, w.l.o.g. assume that $\BA=\BI$, $x_*=0$ and $f(0)=0$ to simplify the notation. Let $\tau=r(1+L) \in(0,\varepsilon / 4)$, and consider the cylinder $T=B_r^{\prime} \times(-3 \tau,-2 \tau)$ where $B'_r$ is the projection of $B_r(x_*)$ into $x_n = 0$. For any $x^{\prime} \in B_r^{\prime}$ we have $f\left(x^{\prime}\right) \geq-r L>-\tau$, and therefore $4\tau\geq f\left(x^{\prime}\right)-x_n \geq \tau$ for all $\left(x^{\prime}, x_n\right) \in T$. Further, since $(3\tau)^2+r^2=(9(L+1)^2+1)r^2< \varepsilon^2$, we obtain that $T$ is still inside  $B_{\varepsilon} \cap V$ as we can see in the figure \ref{fig:wK}. 
\begin{figure}[!h]
    \centering
    \includegraphics[scale = 0.5]{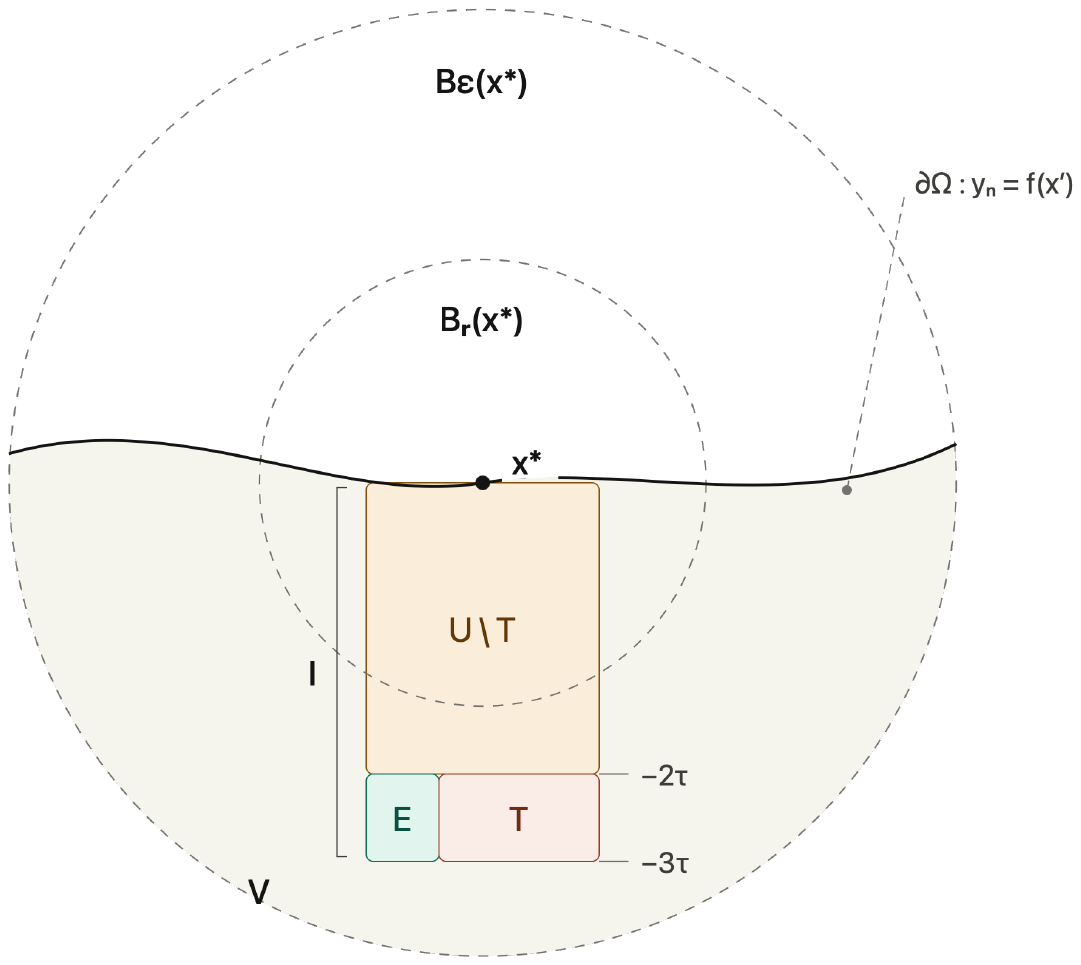}
    \caption{Sketch of the geometry in the construction of Lemma~3.2}
    \label{fig:wK}
\end{figure}
Additionally, since  $\operatorname{diam}(T)\leq c(n)\tau,$ by the standard Poincaré inequality there exists $a \in \mathbb{R}$ such that
\begin{align}
\label{3.5}
\int_T|f(x')-x_n|^{\alpha p}|u-a|^p d x &\leq 4^{\Ga p} \tau^{\Ga p}\int_T|u-a|^p dx\\ \nonumber
&\leq 4^{\alpha p} C(n,p,L) \tau^{(1+\alpha)p} \int_T|\nabla u|^p dx\\ \nonumber
&\leq 4^{\alpha p} C(n,p,L) \int_T|f(x')-x_n|^{(1+\alpha) p}|\nabla u|^p dx.
\end{align}
In the next step we will apply Lemma \ref{Lem:3.1} to $u\left(x^{\prime}, \cdot\right)$ for each $x^{\prime} \in B_r^{\prime}$, with $I=\left(-3 \tau, f\left(x^{\prime}\right)\right)$ and $E=(-3 \tau,-2 \tau)$. Since $\tau=|E|<|I| \leq 4 \tau$, we have
\begin{align}
\label{3.6}
\int_I|f(x')-x_n|^{\alpha p}|u(x',x_n)-a|^p d x_n &\leq C(n,p,\alpha,L)\int_I\left|f\left(x^{\prime}\right)-x_n\right|^{(\alpha+1)p}|\nabla u|^p d x_n\\ \nonumber
&+C(n,p,\alpha,L) \int_E|f(x')-x_n|^{\alpha p}|u(x',x_n)-a|^p d x_n.
\end{align}
Let $U=\left(B_r^{\prime} \times(-3\tau, \infty)\right) \cap V$, so that $B_r(x_*) \cap \Omega=B_r(x_*) \cap U$ and 
$U \subseteq B_{\varepsilon}(x_*) \cap V=B_{\varepsilon}(x_*) \cap \Omega$ 
as we can see in Fgure \ref{fig:wK}. We integrate over $x^{\prime} \in B_r^{\prime}$, 
and use both of the above inequalities to conclude
\begin{align}
\label{3.7}
\int_U|f(x')-x_n|^{\alpha p}|u-a|^p d x 
& \leq C(\Ga,p,n,L) \int_{B_r^{\prime}} \int_I\left|f\left(x^{\prime}\right)-x_n\right|^{(\alpha+1)p}|\nabla u|^p d x_n d x^{\prime}\\ \nonumber
&\quad + C(\Ga,p,n,L) \int_T\left|f\left(x^{\prime}\right)-x_n\right|^{(\alpha+1)p}|\nabla u|^p d x \\ \nonumber
& \leq C(\Ga,p,n,L) \int_U\left|f\left(x^{\prime}\right)-x_n\right|^{(\alpha+1)p}|\nabla u|^p d x.
\end{align}
Additionally, by the construction of $U$ we have that $|f\left(x^{\prime}\right)-x_n|$ is comparable with 
$\delta(x)= \delta_{\partial \GO}(x)$ for $x\in U.$ In particular, in [Lemma5.7, \ref{bib:Con.Zwi.}] it is proven that there exists $C(L)>0$, such that
$$
\left|f\left(x^{\prime}\right)-x_n\right| \leq C(L) \Gd_{\partial \GO}(x)\leq C(L)\delta_\GG(x) \quad \text { for all }\quad x \in U ,
$$
and for a lower bound, we use the fact that $x^*\in \GG$ to get
\begin{align*}
    \Gd_\Gamma(x)&\leq |x-x^{\ast}|\\
    &\leq\sqrt{9\tau^2+r^2}\\
    &\leq 3 |f\left(x^{\prime}\right)-x_n|.
\end{align*}
The proof of the lemma is completed since $B_r(x_{\ast}) \cap V \subseteq U$.
\end{proof}


\section{Proof of Theorem~\ref{Thm:2.4}}
\setcounter{equation}{0}
\label{Sec:4}

In this section, we will prove the weighted Poincaré inequalities in Theorem~\ref{Thm:2.4}. The proof is slightly more involved than the one in [\ref{bib:Aco.Cej.Dur.},\ref{bib:Con.Zwi.}] as a standard boundary plus an interior cover does not work due to the specifics of the distance taken in the inequalities. For bulk domains, will need to construct a uniform $\varepsilon-$scale cover, while for thin domains a nice domain transformation through a suitable change of variables will do the trick. 


\begin{proof}[Proof of Theorem 2.4\textup{(i)}]

It is sufficient to prove the scalar-valued case. For $a,b>0,$ the symbolics $a\simeq b$ means $ca\leq b\leq\frac{1}{c}a$ for some 
$c=c(n,p,\Ga,L,R.)$


\begin{lemma}[Uniform covering of $\Omega$]
\label{Lem:4.1}
Let $\GO\subset\mathbb R^{n}$ be an open, connected $(L,R)$-Lipschitz set as in Definition~\ref{Def:2.1}, and
let $\Gamma\subseteq\partial\GO$ be nonempty and closed. There exist a natural number $N\in\mathbb N$,
open sets $U_0,\dots,U_N\subseteq\GO$, numbers $r_0,\dots,r_N>0$, a map
$\pi:\{1,\dots,N\}\to\{0,\dots,N\}$ with $\pi(i)<i$, and balls $\tilde B_i\subseteq U_i\cap U_{\pi(i)}$
$(1\le i\le N)$, with the following properties.
\begin{enumerate}
\item $\GO=\bigcup_{i=0}^{N}U_i$.
\item $N$ depends only on $n$, $L$ and $R,$ and $r_i \simeq \eps$ for every $i\in \{1,\dots,N\}$.
\item For every $i\in \{0,\dots,N\}$ one has either (a) or (b) below, where
\begin{enumerate}
\item $U_i=\GO\cap B_{r_i}(x_i^{*})$ for some $x^{*}_i\in\Gamma$ and $r_i\le\eps/(5+4L)$, so that
Lemma 3.2 applies on $U_i$.
\item $\dist(U_i,\GG)\simeq r_i$, $\diam U_i\simeq r_i$, and $U_i$ supports the standard Poincaré
inequality with a uniform constant $C=c(n,p,L)\,r_i.$
\end{enumerate}
\item For every $1\leq i\leq N$ one has $\mathrm{diam} \tilde B_i \simeq r_i.$
\end{enumerate}
\end{lemma}


\begin{proof}[Proof of Lemma~\ref{Lem:4.1}]
We divide the construction into several steps. \\
\noindent\textbf{Claim 1: Boundary sections contain comparable balls.} \textit{For $c_0=\dfrac{1}{5(1+L)^{2}}$ and for every $z\in\GO$ and every $t\in(0,\eps/2],$ there is $y\in\GO$
such that $B_{c_0t}(y)\ \subseteq\ B_t(z)\cap\GO.$}\\
\begin{proof}[Proof of Claim 1]
This is a direct consequence of the domain Lipschitz property. Indeed, we have that if  $\Gd_{\partial \GO}(z)\ge t/2$, then just
take $y=z$. Assume now $\Gd_{\partial \GO}(z)<t/2$ and pick $w\in\partial\GO$ with $|w-z|=\Gd_{\partial \GO}(z)$. 
After an isometry we have $w=0$, $f(0)=0$, $\mathrm{Lip} f\le L$ and $B_\eps(0)\cap\GO=B_\eps(0)\cap V$ with
$V=\{y_n<f(y')\}$. Put $\Gl=\frac{1}{4(1+L)}$ and $y=z-\lambda te_n .$
Then $|y-z|+c_0t=\lambda t+c_0t<t$, so $B_{c_0t}(y)\subseteq B_t(z)$. Also we have $|y-w|+c_0t\le t/2+\lambda t+c_0t<t\le\eps/2,$
so $B_{c_0t}(y)\subseteq B_\eps(w)$. Finally, let $v\in B_{c_0t}(y)$. Since $z\in\GO\cap B_\eps(w)=V\cap B_\eps(w)$ we have $f(z')>z_n$, and therefore
\begin{align*}
f(v')-v_n & \ge \big(f(z')-L|v'-z'|\big)-\big(z_n-\lambda t+c_0t\big)\\
&>\lambda t-(1+L)c_0t \\
&>0,
\end{align*}
hence $B_{c_0t}(y)\subseteq V\cap B_\eps(w)=\GO\cap B_\eps(w)$. Claim 1 is proven. 
\end{proof}

\noindent\textbf{Claim 2: Boundary Poincar\'e sets.} \textit{Let $z\in\partial\GO$ and $0<\sigma\le\frac{2\varepsilon_L}{5}.$ In the frame of Definition 2.1 at $z$ define the
collar
$$
D(z,\sigma):=\big\{(y',y_n) \ : \ |y'|<5\sigma,\ f(y')-5(1+L)\sigma<y_n<f(y')\big\}.
$$
Then the below statements hold:
\begin{itemize}
\item[(i)] $\GO\cap B_{4\sigma}(z)\ \subseteq\ D(z,\sigma)\ \subseteq\ \GO \cap B_{10(1+L)\sigma}(z)$.
\item[(ii)] $D(z,\sigma)=\Phi(W)$ where $W=B'_{5\sigma}\times(-5(1+L)\sigma,0)$ and
the mapping $\Phi(y',t)=(y',t+f(y'))$ satisfies $\det D\Phi=1$ and \  $\mathrm{Lip}\Phi, \mathrm{Lip}\Phi^{-1} \le1+L$.
Consequently $D(z,\sigma)$ is open and connected, $\diam D(z,\sigma)\simeq\sigma$, and $D(z,\sigma)$ supports the standard Poincar\'e inequality. 
\end{itemize}
}
\begin{proof}[Proof of Claim 2]

(i) If $y\in\GO\cap B_{4\sigma}(z)$ then $y\in V$, i.e.\ $y_n<f(y')$, and $|y'|\le4\sigma,$
while
\begin{align*}
f(y')-y_n&=\big(f(y')-f(0)\big)+\big(0-y_n\big)\\
&\le L|y'|+|y_n|\\
&<4(1+L)\sigma .
\end{align*}
Conversely, if $y\in D(z,\sigma)$ then
\begin{align*}
|y| & <|y'|+|f(y')|+5(1+L)\sigma \\
& =10(1+L)\sigma \\ 
& \le\eps, 
\end{align*}\
thus  $y\in V\cap B_\eps(z)=\GO\cap B_\eps(z)$.

(ii) Obviously $\det D\Phi=1$ and $|\Phi(a)-\Phi(b)|\le(1+L)|a-b|$, and the same for
$\Phi^{-1}(y',y_n)=(y',y_n-f(y'))$. Now the Poincaré inequality on the cylinder $W$ transfers it to a 
Poincaré inequality on $D$ by the change of variables through $\Phi.$ Claim 2 is proven.
\end{proof}

We now return to the proof of Lemma~\ref{Lem:4.1}. Set 
$$
\rho=\frac{\varepsilon_L}{4},\quad s=\frac{\rho}{2(1+L)},\quad \eta=\frac s4.
$$
Lemma 3.2 is applicable on $\GO\cap B_{4\rho}(x^{*})$ for every $x^{*}\in\Gamma,$ and Claim 2 is applicable with $\sigma=s$. Define the cover of $\GG$ as follows: take $\{x^{*}_1,\dots,x^{*}_{N_1}\}$ to be a maximal $\rho$-separated subset of $\Gamma,$ and define
\begin{equation}
\label{4.1}
A_i^{\circ}:=\GO\cap B_{3\rho}(x^{*}_i),\qquad A_i:=\GO\cap B_{4\rho}(x^{*}_i),\qquad
r(A_i):=4\rho,\quad i=1,\dots N_1.
\end{equation}
Now we define the cover of the remaining part of the boundary. Let $\gamma=\{z\in\partial\GO \ :\ \Gd_\GG(z)\ge\rho\}$ and let $\{z_1,\dots,z_{N_2}\}$ be a maximal $s$-separated subset of $\Gg$ (possibly empty, e.g. when $\Gamma=\partial\GO$). Define
\begin{equation}
\label{4.2}
C_j^{\circ}=\GO\cap B_{2s}(z_j),\qquad C_j=D(z_j,s),\qquad r(C_j)=s,\quad  j=1,\dots N_2.
\end{equation}
Now we define the Interior sets. Let $E=\{x\in\GO:\ \Gd_{\partial \GO}(x)\ge s\ \text{and}\ \Gd_\GG(x)\ge2\rho\}$ and let
$\{w_1,\dots,w_{N_3}\}$ be a maximal $(s/2)$-separated subset of $E$ (possibly empty). Put
\begin{equation}
\label{4.3}
B_k^{\circ}=B_{s/2}(w_k),\qquad B_k=B_{s}(w_k),\qquad r(B_k)=s,\quad  k=1,\dots N_3.
\end{equation}
All the sets and cores are open, all the cores are contained in the corresponding sets, and all the sets are contained in $\GO$. We relabel the family $\{A_i\}\cup\{C_j\}\cup\{B_k\}$ as $U_0,\dots,U_N,$ the core sets as $U_i^{\circ},$ and set $r_i=r(U_i)$ with the proper ordering to be defined later in the proof. Now, part 1 of the lemma easily follows from the definition of the sets $U_i,$ moreover we have $\GO=\cup U_i^{\circ}.$ Part 2 of the lemma follows from the fact that $\diam \overline\GO<R\eps,$ the construction, utilizing Claim 1. Part 3 of the lemma follows from Claim 2 and the construction. It is straightforward (utilizing Claim 1), that
\begin{equation}
\label{4.4}
w\in U_i^{\circ} \qquad\text{implies }\qquad \ B_{\eta}(w)\cap\GO\ \subseteq\ U_i.
\end{equation}
 Let now $G$ be the graph with the vertex set $\{0,\dots,N\}$ and an edge between $j$ and $k$ whenever
$U^{\circ}_j\cap U^{\circ}_k\ne\emptyset$. Let's prove that $G$ is connected. Assume in contradiction that it is not, and let $V_1$ be the vertex set of one connected component, $V_2$ the rest, both nonempty. Then the nonempty open sets
$O_1=\bigcup_{j\in V_1}U^{\circ}_j$ and $O_2=\bigcup_{j\in V_2}U^{\circ}_j$ would separate $\Omega,$ a contradiction. 
Because any connected graph has a spanning tree [\ref{bib:Bollobas}], we can choose a spanning tree of $G,$ label the root as $U_0,$  and enumerate the vertices in order of nondecreasing distance to the root. This automatically defines the mapping $\pi(i)$ as well, according to the rule that vertex $\pi(i)$ is the parent of vertex $i$ in the tree.  Finally we define the the overlap sets $\tilde B_i$. 
Fix $i\in \{1,\dots,N\}.$ There exists $w\in U^{\circ}_i\cap U^{\circ}_{\pi(i)}$, thus by (\ref{4.4}) we have 
$B_\eta(w)\cap\GO\ \subseteq\ U_i\cap U_{\pi(i)}.$ Claim 1 with $t=\eta\ (\le\eps/2)$ yields $y\in\GO$ with 
$B_{c_0\eta}(y)\subseteq B_\eta(w)\cap\GO$. Defining $\tilde B_i=B_{c_0\eta/2}(y)$ it is straightforward to check that 
$\tilde B_i\subseteq\ U_i\cap U_{\pi(i)}$ and the remaning parts of the lemma concerning $\tilde B_i.$ 
The proof of the lemma is complete now.

\end{proof}


We now prove Theorem~\ref{Thm:2.4}. For the sake of brevity, all constants $C,c>0$ in the proof of the theorem may depend only on $n,p,\Ga,L$ and $R.$ Set  $M=\|(\Gd_\GG)^{1+\alpha}\nabla u\|_{L_p(\Omega)}.$ For each set $U_i$, choose a constant $a_i\in\mathbb R$ as follows: If $U_i$ is a boundary coordinate set centered at a point of $\Gamma$, apply Lemma~3.2, while if $U_i$ is an interior set, let $a_i$ be the constant in the Poincar\'e inequality for $u$ on $U_i$. In either case we have
\begin{equation}
\label{4.5}
\|(\Gd_\GG)^\alpha(u-a_i)\|_{L_p(U_i)} \leq CM.
\end{equation}
Indeed, on a boundary set we apply Lemma~3.2 and on an interior set $U_i$, one has $\Gd_\GG\simeq r_i$, and therefore the standard Poincar\'e inequality yields
\begin{align*}
\|(\Gd_\GG)^\alpha(u-a_i)\|_{L_p(U_i)} &\leq C r_i^\alpha \|u-a_i\|_{L_p(U_i)} \\
&\leq C r_i^{1+\alpha}\|\nabla u\|_{L_p(U_i)} \\
&\leq C\|(\Gd_\GG)^{1+\alpha}\nabla u\|_{L_p(U_i)}\\
&\leq CM
\end{align*}
We now compare the constants $a_i$ associated with neighboring sets. Since $ \Gd_\GG(x)\simeq r_i$ on $\tilde B_i,$ and 
$|\tilde B_i|\geq c r_i^n,$ we have $\|(\Gd_\GG)^\alpha\|_{L_p(\tilde B_i)} \geq c r_i^{\alpha+n/p}.$ Hence
\begin{align}
\label{4.6}
r_i^{\alpha+n/p}|a_i-a_{\pi(i)}| &\leq C\|(\Gd_\GG)^\alpha(a_i-a_{\pi(i)})\|_{L_p(\tilde B_i)} \nonumber\\
&\leq C\|(\Gd_\GG)^\alpha(u-a_i)\|_{L_p(U_i)} + C\|(\Gd_\GG)^\alpha(u-a_{\pi(i)})\|_{L_p(U_{\pi(i)})}.
\end{align}
Choose now $a=a_0.$ Iterating (\ref{4.6}) along the parent chain connecting $U_i$ to $U_0$, and using the fact that the length of every such chain is bounded by a constant depending only on $n,L,$ and $R$, we obtain
\begin{equation}
\label{4.7}
r_i^{\alpha+n/p}|a_i-a| \leq CM.
\end{equation}
Since $\Gd_\GG\leq Cr_i$ on $U_i$, and $|U_i|\leq Cr_i^n$, it follows from (\ref{4.7}) that
\begin{equation}
\label{4.8}
\|(\Gd_\GG)^\alpha(a_i-a)\|_{L_p(U_i)} \leq C r_i^{\alpha+n/p}|a_i-a|\leq CM.
\end{equation}
Combining (\ref{4.5})-(\ref{4.8}) we arrive at
\begin{align*}
\|(\Gd_\GG)^\alpha(u-a)\|_{L_p(\Omega)}^p &\leq C \sum_{i=0}^N \|(\Gd_\GG)^\alpha(u-a)\|_{L_p(U_i)}^p \\
&\leq C \sum_{j=0}^N \|(\Gd_\GG)^\alpha(u-a_i)\|_{L_p(U_i)}^p+C  \sum_{i=0}^N \|(\Gd_\GG)^\alpha(a_i-a)\|_{L_p(U_i)}^p \\
&\leq CM^p.
\end{align*}
Therefore,
$$
\|(\Gd_\GG)^\alpha(u-a)\|_{L_p(\Omega)}
\leq C\|(\Gd_\GG)^{1+\alpha}\nabla u\|_{L_p(\Omega)}.
$$
This completes the proof of part (i) of the theorem. 
\end{proof}


\begin{proof}[Proof of Theorem 2.4\textup{(ii)}]

We prove the weighted Poincar\'e inequality on the thin domain by reducing the problem to a
uniformly Lipschitz reference domain. Recall that
$$
\Omega_h=\left\{x+t\,n(x): x\in S,\; -g_1^h(x)<t<g_2^h(x) \right\}.
$$
Define the normalized thickness functions
$$
a_h(x)=\frac{g_1^h(x)}{h}, \qquad b_h(x)=\frac{g_2^h(x)}{h}.
$$
By the assumptions on $g_1^h$ and $g_2^h$ we have
$$
1\leq a_h(x),b_h(x)\leq c_1 \quad \text{and}\quad |\nabla a_h(x)|+|\nabla b_h(x)|\leq c_2\quad \text{for all}\quad  x\in S.
$$
Under (H) and (\ref{2.3}), there exists $h_0>0$, depending only on $S$ and $c_1$, such that for $h\in(0,h_0)$ the map
$$
\Psi(x,t)=x+t\,n(x),\qquad x\in S,\ |t|\le c_1h,
$$
is injective and bi-Lipschitz with controllable constants onto its image. In particular the thin domain
$\GO_h=\Psi \{(x,t):-g_1^{h}(x)<t<g_2^{h}(x)\}$ satisfies $|\GO_h|\simeq h,$ $\diam\GO_h\simeq\diam S$, and
\begin{equation}
\label{4.9}
c \dist(x,\partial S)\le\Gd'(\Psi(x,t))\le C \dist(x,\partial S)\quad\text{for all}\quad x\in S, \  t\in (-g_1^{h}(x),g_2^{h}(x)).  
\end{equation}
Indeed, due to the compactness and $C^2-$regularity of $S,$ there exists $C=C(S)>0$ such that
\begin{equation}
\label{4.10}
|(x_1-x_2)\cdot n(x_1)|\leq C|x_1-x_2|^2,  \  \ |n(x_1)-n(x_2)|\leq C|x_1-x_2| \quad \text{for all}\quad x_1,x_2\in S.
\end{equation}
We assume $h_0\leq \frac{1}{4Cc_1}.$ For $x_1,x_2\in S$ and for $|t_1|,|t_2|\leq c_1h_0,$ we have for $w=x_1+t_1n(x_1)-x_2-t_2n(x_2)$ utilizing (\ref{4.10}), that
\begin{align*}
|w||x_1-x_2| & \geq |w \cdot (x_1-x_2)| \\
& \geq |x_1-x_2|^2-c_1h_0|n(x_1)\cdot (x_1-x_2)|-c_1h_0|n(x_2)\cdot (x_1-x_2) |\\
&\geq \frac{1}{2}|x_1-x_2|^2,
\end{align*}
thus $|w| \geq \frac{1}{2}|x_1-x_2|.$ We have on the other hand 
\begin{align*}
|w| &\geq |w\cdot n(x_1)| \\
& \geq |t_1-t_2|-|t_2||(n(x_1)-n(x_2))\cdot n(x_1)|-|n(x_1)\cdot (x_1-x_2)|  \\
&\geq |t_1-t_2|-(1+Cc_1h_0)|x_1-x_2|\\
&\geq |t_1-t_2|-\frac{5}{4}|x_1-x_2|.
\end{align*}
Combining both estimates we get after some algebra $|w| \geq \frac{2}{11}\left(|x_1-x_2|+|t_1-t_2|\right).$
The upper bound $|w|\leq C\left(|x_1-x_2|+|t_1-t_2|\right)$ is straightforward. We now introduce the reference domain
$$
\widehat{\Omega}_{h_0}^h=\Psi \{(x,t) \ : \  \ x\in S, \  -h_0a_h(x)<t<h_0b_h(x)\}.
$$
We will briefly justify next that the domains $\widehat{\Omega}_{h_0}^h$ form a uniformly Lipschitz family. 
The point is that $\widehat\GO_{h_0}^h$ has thickness of order $h_0$ and that $h$ enters only through the
normalized thickness functions $a_h,b_h$, which have bounded Lipschitz constants uniformly in $h$ by (\ref{2.3}). 
This is necessary for the derivation of part (ii) of Theorem~\ref{Thm:2.4} from part (i).

\begin{lemma}[Uniformly Lipschitz reference domains]
\label{Lem:4.2}
Assume \textup{(H)} and \textup{(\ref{2.3})}, and let $h_0=h_0(S,c_1)$ be as above. After possibly decreasing
$h_0$ by a factor depending only on $c_2$ and $L_S$, there exist constants
$\eps_0,L_0,R_0>0$, depending only on $S,c_1,c_2$ and $L_S$, such that
for every $h\in(0,h_0)$ the domain $\widehat\GO_{h_0}$ is $(L_0,R_0)$-Lipschitz in the sense
of Definition~\ref{Def:2.1}, with the parameter $\eps_0$.
\end{lemma}


\begin{proof}

Recall that
$$
1\leq a_h,b_h\leq c_1,\qquad |\nabla a_h|+|\nabla b_h|\leq c_2.
$$
Fix one of the finitely many coordinate charts supplied by hypothesis (H). In that chart, write the surface as
$$
x=\Gth(y),\qquad y\in O\subset\mathbb R^2,
$$
where $O$ is either a disk or, near $\partial S$ it is a uniformly Lipschitz subgraph. Consider first the parameter domain
$$
G_h=\left\{(y,s) \ : \  y\in O,  \ -h_0a_h(\Gth(y))<s <h_0b_h(\Gth(y)) \right\}.
$$
The upper and lower boundary functions have uniformly bounded Lipschitz constants, since
$$
\mathrm{Lip}\bigl(h_0a_h\circ\Gth\bigr)+\mathrm{Lip}\bigl(h_0b_h\circ\Gth\bigr) \leq C(S)h_0c_2.
$$
Near the lateral boundary, the base $O$ has Lipschitz constant bounded by $L_S$. Hence $G_h$ is a Lipschitz domain with constants independent of $h$. The same conclusion holds at the intersections of the lateral and upper or lower boundaries, because the face slope $C(S)c_2h_0$ is small relative to the transversality threshold, so that a tilted direction, e.g., 
 $-\kappa e_3-e_2$ where $\kappa\simeq \frac{1}{1+L_S}$ sees both constraints as graphs. Now define
$$
\Theta(y,s)=\Gth(y)+s n(\Gth(y)).
$$
By the previously proved tubular-neighborhood estimate, after decreasing $h_0$ if necessary, $\Theta$ and $\Theta^{-1}$ are bi-Lipschitz on
$$
\{(y,s) \ : \ y\in O,\ |s|\leq c_1h_0\},
$$
with constants depending only on $S$ and $c_1$. Also, ince the normal vector field $n(x)$ is $C^1$, the transform $\Theta$ is locally and thus globally $C^1$ too. We can expand $\Theta$ at a fixed point $y_0$ as:
$$\Theta(y)=\Theta(y_0)+D \Theta(y_0)\circ \GL(y,y_0),$$
where $\GL(y,y_0)$ is close to the translated identity map when $y$ is close to $y_0.$ Close to the identity maps and fixed invertible linear maps (as $D\Theta(y_0)$ here) preserve Lipschitz graphs quantitatively, thus $\Theta(y)$ will do so too. Consecuently, since 
$ \Theta(G_h)= \widehat{\GO}_{h_0}^h $ within the corresponding surface patch, the uniform Lipschitz character of $G_h$ transfers through $\Theta$ to $\widehat{\Omega}_{h_0}^h$. Finally, the atlas of $S$ is finite and
$$
\operatorname{diam}\widehat{\Omega}_{h_0}^h
\leq
\operatorname{diam}S+2c_1h_0.
$$
Consequently, there exist constants $L_0,R_0>0$, depending only on $S,c_1,c_2$ and the fixed choice of $h_0$, 
such that every $\widehat{\Omega}_{h_0}^h$ is an $(L_0,R_0)$-Lipschitz domain with a parameter $\eps_0$ by compactness. The lemma is now proven.

\end{proof}

Define next
$$
\Phi_h\colon \widehat{\Omega}_{h_0}\longrightarrow\Omega_h\quad\text{by}\quad \Phi_h(x+sn(x))=x+\frac{sh}{h_0}n(x).
$$
For $h\in (0,h_0),$ the map $\Phi_h$ is bi-Lipschitz, with constants uniform in $h$ as proven above.
 Furthermore, its Jacobian satisfies
\begin{equation}
\label{4.11}
c h \leq |\det D\Phi_h(x,s)| \leq C h
\end{equation}
for all $x\in S$ and  $-h_0a_h(x)<s<h_0b_h(x).$ Set
$$
\widehat{\Gamma}_{h_0}=\left\{ x+sn(x) \ : \ x\in\partial S,\; -h_0a_h(x)\leq s\leq h_0b_h(x) \right\},
$$
and 
$$
\widehat d_{h_0}(x+tn(x))= \dist\bigl(x+tn(x),\widehat{\Gamma}_{h_0}\bigr)\quad\text{for}\quad x\in S,  \  -h_0a_h(x)<t<h_0b_h(x). 
$$
Uniformly in $h$ for small enough $h>0$ we have by the bi-Lipscitz property of $\Psi:$
\begin{equation}
\label{4.12}
c\,\widehat d_{h_0}(x+sn(x)) \leq \delta'\bigl(\Phi_h(x+sn(x))\bigr) \leq C\,\widehat d_{h_0}(x+sn(x))
\quad\text{for}\quad x\in S,  \  -h_0a_h(x)<t<h_0b_h(x). 
\end{equation}
Let $u\in W^{1,p}(\Omega_h;\mathbb R^3)$, and set
$$
v(x+sn(x))=u\bigl(\Phi_h(x+sn(x))\bigr)\quad\text{for}\quad x\in S,  \  -h_0a_h(x)<t<h_0b_h(x). 
$$
By part (i) of Theorem~\ref{Thm:2.4}, applied to the uniformly Lipschitz domain $\widehat{\Omega}_h$ and to the closed boundary subset $\widehat{\Gamma}_{h_0}$, there exists $a\in\mathbb R^3$ such that
\begin{equation}
\label{4.13}
\|\widehat d_{h_0}^\alpha(v-a)\|_{L^p(\widehat{\Omega}_{h_0})} 
\leq C \|\widehat d_{h_0}^{1+\alpha}Dv\|_{L^p(\widehat{\Omega}_{h_0})},
\end{equation}
uniformly for small $h$. Next we estimate $Dv$. For $y=x+sn(x)$ we have
$$
Dv(y) = \nabla u\bigl(\Phi_h(y)\bigr)D\Phi_h(y).
$$
Since $D\Phi_h$ is uniformly bounded for small $h$, we obtain
\begin{equation}
\label{4.14}
|Dv(y)| \leq C \left| \nabla u\bigl(\Phi_h(y)\bigr) \right|.
\end{equation}
Using (\ref{4.11}) and (\ref{4.12}), and the change of variables $z=\Phi_h(y)$, the left-hand side of (\ref{4.13}) satisfies
\begin{align*}
\|(\delta')^\alpha(u-a)\|_{L^p(\Omega_h)}^p 
&= \int_{\widehat{\Omega}_{h_0}} \delta'\bigl(\Phi_h(y)\bigr)^{\alpha p} |v(y)-a|^p |\det D\Phi_h(y)|dy \\
& \leq C h \int_{\widehat{\Omega}_{h_0}} \widehat d_{h_0}(y)^{\alpha p} |v(y)-a|^pdy\\
&= C h \|\widehat d_{h_0}^\alpha(v-a)\|_{L^p(\widehat{\Omega}_{h_0})}^pdy
\end{align*}
Similarly, by (\ref{4.11}) and (\ref{4.14}) we have that 
\begin{align*}
h \|\widehat d_{h_0}^{1+\alpha}Dv\|_{L^p(\widehat{\Omega}_{h_0})}^p
&\leq C h \int_{\widehat{\Omega}_{h_0}} \widehat d_{h_0}(y)^{(1+\alpha)p} \left| \nabla u\bigl(\Phi_h(y)\bigr) \right|^p dy \\
&\leq C \int_{\Omega_h} (\delta'(y))^{(1+\alpha)p} |\nabla u(z)|^p dz \\
&= C \|(\delta')^{1+\alpha}\nabla u\|_{L^p(\Omega_h)}^p.
\end{align*}
Combining these estimates with (\ref{4.13}), we conclude that
$$
\|(\delta')^\alpha(u-a)\|_{L^p(\Omega_h)}
\leq
C
\|(\delta')^{1+\alpha}\nabla u\|_{L^p(\Omega_h)}.
$$
This completes the proof of part (ii) of the theorem.
\end{proof}


\begin{remark}
\label{Rem:4.2} 
Even if $u=0$ on $\GG,$ one can't take $a=0$ as in the standard Poincaré inequality with vanishing boundary conditions. This happens because to control the function near the boundary we used information from the interior. 
To confirm this, we consider the following 1d counter-example for $\Omega =(0,1)$ and $\GG=\{0\}$. Take

$$u_n=\begin{cases} 
n & if  \quad \delta_\GG(x)\geq \frac{1}{n}\\
\delta_\GG(x)n^{2} &  if \quad   \delta(x)\leq \frac{1}{n} 
\end{cases}
\qquad\text{then}\qquad
 |u'_n| = 
\begin{cases} 
0 &  if  \quad  \delta_\GG(x)> \frac{1}{n}\\
n^{2} & if \quad  \delta_GG(x)\leq \frac{1}{n}
 \end{cases}
$$
So for $n>4$ we have that 
$$\int_0^1 \delta_\GG(x)^{\alpha p} u_n(x) ^p \geq \int_{\delta_\GG(x)\geq 1/4} \delta_\GG(x)^{\alpha p}n^p \geq \frac{1}{4^{\alpha p}}\frac{1}{2}n^p\to \infty$$
However $$\int_0^1 \delta_\GG(x)^{(\alpha+1) p} |u'_n(x)| ^p = \int_0^{\frac{1}{n}} x^{(\alpha+1)p} n^{2p} =\frac{2}{(\alpha+1)p+1}n^{2p-(\alpha+1)p-1} =C n^{p(1-\alpha )-1} $$
so we conclude that for any $p\geq 1$ and $\alpha\geq 0:$ 
$$\frac{\|\delta(x)^\alpha |u_n(x)|\|^p}{\|\delta(x)^{\alpha+1} |u'_n(x)|\|^p}\geq C \frac{n^p}{n^{p(1-\alpha)-1}}=n^{1+p\alpha}  \to \infty.$$

\end{remark}


\section{Proof of Theorem~\ref{Thm:2.2}}
\setcounter{equation}{0}
\label{Sec:5}

In this section, we will prove the weighted inequalities in Theorem~\ref{Thm:2.2}. The idea is similar to the proof of the weighted Poincaré inequality (Theorem~\ref{Thm:2.4}), but the covering part is a more delicate infinite Whitney type cover. This kind of estimates are typically proven for $\GG=\dOm$, but we will prove it for any $\GG\subset\dOm$ because it will be necessary to prove the inequality for thin domains. We divide the proof into several parts. First we prove a Whitney-type decomposition of $\GO$ adapted 
to $\Gamma$.


\begin{lemma}
\label{Lem:5.1}
Let $\GO\subset\mathbb R^{n}$ be an open connected $(L,R)$-Lipschitz set with parameter $\varepsilon>0$, and let
$\Gamma\subseteq\partial\GO$ be nonempty and closed. There exist constants $\lambda\ge1$, 
$M\in\mathbb N$, $L_0,R_0>0$, all depending only on $n,L,R$, a countable family of open sets $U_i\subset\GO$,
numbers $r_i>0$ and functions $\varphi_i\in C^{\infty}(\GO;[0,1])$ with the indexing set $J$ such that
\begin{enumerate}
\item $\GO=\bigcup_{i\in J}U_i$, each $U_i$ is connected, $(L_0,R_0)$-Lipschitz and
$\diam U_i\le\lambda r_i.$
\item $\lambda^{-1}r_i\le\Gd_\GG(x)\le\lambda r_i$ for all $x\in U_i.$
\item $\sum_i\chi_{U_i}\le M$ on $\GO$.
\item $\operatorname{supp}\varphi_i\cap\GO\subset U_i$, $\ \sum_i\varphi_i\equiv1$ on $\GO$,
$\ |\nabla\varphi_i|\le\lambda/r_i$.
\end{enumerate}
\end{lemma}


\begin{proof}[Proof of Lemma~\ref{Lem:5.1}]
Let $m=m(n,L)\geq 2\sqrt n$ yet to be specified, and $k=k(n,L)\ge4\sqrt n m$, and set $\rho=\varepsilon/2m$.  Let $\{Q_i\}$ be a Whitney decomposition of the open set  $\mathbb R^{n}\setminus\GG$ into closed dyadic cubes with pairwise disjoint interiors such that
\begin{equation}
\label{5.1}
k\ell(Q_i) \le \dist(Q_i,\Gamma) \le 10k\sqrt n\ell(Q_i),
\end{equation}
where $\ell(Q)$ is the side length of $Q.$ This is possible by taking the standard decomposition of $\mathbb R^{n}\setminus\GG$ (see [\ref{bib:Stein}] for details), which satisfies $\diam Q\le\dist(Q,\Gamma)\le4\diam Q$, and subdividing each cube dyadically $q$ times with $2^q\ge k>2^{q-1}$. We subdivide every $Q_i$ with $\ell(Q_i)>\rho$ into congruent subcubes of side length in $(\rho/2,\rho]$, and keep only those cubes $Q$ of the resulting family with 
$\tfrac32Q\cap\GO\ne\emptyset.$ Set $r_i=\ell(Q_i).$ For every retained cube we have 
\begin{equation}
\label{5.2}
kr_i\le\dist(Q_i,\Gamma)\le K r_i,\quad \text{uniformly for some} \quad K=K(n,L,R).
\end{equation}
Indeed, the lower bound is obvious. For cubes that were not subdivided this is (\ref{5.1}). For a subdivided one we have 
 $r_i>\rho/2,$ and since $\tfrac32Q_i$ meets $\GO$ and $\Gamma\subset\partial\GO$, we have
 \begin{align*}
\dist(Q_i,\Gamma)&\le\diam\GO+\diam Q_i\\
& \leq R\eps+\sqrt nr_i\\
& \leq (4Rm+\sqrt n)r_i.
\end{align*}
We now construct the open sets $U_i.$ For any retained cube $Q_i$ consider two cases.\\
\textbf{Case (a). $2Q_i\cap\partial\GO=\emptyset$.} Since $\tfrac32Q_i$ meets $\GO$ and $2Q_i$ is
connected and disjoint from $\partial\GO$, we have $2Q_i\subset\GO$. Put $U_i=\operatorname{int}(2Q_i),$
which is $(L_0,R_0)$-Lipschitz with absolute constants, and $\tfrac32Q_i\cap\GO\subset U_i$.\\
\textbf{Case (b). $2Q_i\cap\partial\GO\ne\emptyset$.} Pick $x_i\in2Q_i\cap\partial\GO$ and use Definition 2.1 at $x_i.$ We have up to an isometry that $x_i=0$, $f(0)=0$ and
$$B_\eps(0)\cap\GO=B_\eps(0)\cap V,\quad V=\{y_n<f(y')\},\quad \mathrm{Lip}(f)\le L.$$ 
Set
\begin{equation}
\label{5.3}
U_i=\big\{(y',y_n):\ |y'|<3\sqrt nr_i,\quad f(y')-3\sqrt n(1+L)r_i<y_n<f(y')\big\}.
\end{equation}
Now we prove that following properties of $U_i$ in Case (b).
\begin{itemize}
\item[(P1)] $U_i\subset B_{\varepsilon}(x_i)\cap V=B_\varepsilon(x_i)\cap\GO$.
\item[(P2)] $U_i$ is $(L_0,R_0)$-Lipschitz with $L_0,R_0$ depending only on $n,L$, with $\diam U_i\le C(n,L)r_i.$
\item[(P3)] One has $\tfrac32Q_i\cap\GO\subset U_i.$ 
\end{itemize}

\begin{proof}[Proof of (P1)-(P3).]
(P1). We have that every $y\in U_i$ satisfies the bounds
$$
|y-x_i|\le C(n,L)r_i\le C(n,L)\rho\le\eps/2,
$$
as long as $m\ge2C(n,L).$ This fixes the number $m=m(n,L)$ and proves (P1).\\
(P2). The set $U_i$ is the region between two parallel $L$-Lipschitz graphs over a ball, hence connected and $(L_0,R_0)$-Lipschitz with $L_0,R_0$ depending only on $n,L$ (as it is shown in Claim 2 in the proof of Lemma~\ref{Lem:4.1}). The bound 
$\diam U_i\le C(n,L)r_i$ follows from the definition of $U_i$ and the Lipschitz property of $f$ with Lipschitz constant $L.$\\
(P3). Assume $y\in\tfrac32Q_i\cap\GO.$  We have
$$
|y|=|y-x_i|\le\diam(2Q_i)=2\sqrt nr_i,
$$
so $|y'|<3\sqrt nr_i$. Moreover we have 
$$
|y|\le 2\sqrt nr_i \leq 2\sqrt n\rho=\frac{\sqrt{ n} \varepsilon}{m}<\varepsilon,
$$
hence we have $y_n<f(y')$ because $y\in\GO\cap B_\Ge(0)=V\cap B_\Ge(0)$. We have furthermore 
$$
f(y')-y_n=\big(f(y')-f(0)\big)+\big(0-y_n\big)\le2\sqrt n(1+L)r_i.
$$
Hence we have $y\in U_i$ by the definition of $U_i.$

\end{proof}
  
Consequently, we have that in both cases $U_i\subset mQ_i$ (enlarging $m$ if necessary), so parts 1 and 2 of the lemma hold by (\ref{5.2}), with some $\lambda=\Gl(n,L,R)$, because  $\GO=\bigcup_i(\tfrac32Q_i\cap\GO)\subset\bigcup_iU_i\subset\GO$. 

We now prove part 3 of the lemma. Let $x\in\GO$ and let $i$ be such that $x\in U_i\subset mQ_i$. By (\ref{5.2}) 
we have that $r_i\simeq \Gd_\GG(x)$ with constants depending only on $n,L,R$, and $Q_i\subset B(x,C\Gd_\GG(x))$. Since the $Q_i$ have disjoint interiors and $|Q_i|=\ell_i^{n}\simeq \Gd_\GG(x)^{n}$, the number of such $i$ is at most $M=M(n,L,R)$. This proves part 3 of the lemma.

We now construct the partition of unity satisfying part 4 of the lemma.  Choose 
$\Gf_i\in C_c^{\infty}\left(\operatorname{int}\left(\tfrac32Q_i\right);[0,1]\right)$ with $\Gf_i=1$ on $Q_i$ and $|\nabla\Gf_i|\le C(n)/r_i$. 
Then set $\Gf=\sum_i\Gf_i\ge1$ onv$\mathbb R^{n}\setminus\Gamma\supset\GO$, $\Gf\le M$, and neighboring cubes have comparable side lengths, so $|\nabla\Gf |\le C(n,L,R)/r_i$ on $\tfrac32Q_i$. Put $\varphi_i=\Gf_i/\Gf$ on $\GO$ (cubes that
were discarded contribute $0$ on $\GO$). Then $\sum_i\varphi_i=1$ on $\GO$,
$\operatorname{supp}\varphi_i\cap\GO\subset\tfrac32Q_i\cap\GO\subset U_i$ and
$|\nabla\varphi_i|\le C/r_i$. \\
This completes the proof of the lemma.
\end{proof}


We now get back to the proof of Theorem~\ref{Thm:2.2}.

\begin{proof}[Proof of Theorem 2.2] 
We will prove only (\ref{2.2}). The estimate (\ref{2.1}) is obtained from (\ref{2.2}) through a well-known argument by testing (\ref{2.2}) with the field $x+\mu u(x)$ and sending $\mu$ to zero [\ref{bib:Fri.Jam.Mue.1},\ref{bib:Fri.Jam.Mue.2}]. In the proof, the constant $C$ may only depend on $n,p,\alpha, L,R.$  Set $E=\|\Gd_\GG^{\alpha}\dist(\nabla u,\SO(n))\|_{L^{p}(\GO)}$ and assume $E<\infty$, otherwise there is nothing to prove. The rest of the proof below follows the lines in the proof of the geometric rigidity estimate in [\ref{bib:Con.Zwi.}]. We present the details for the convenience of the reader. For any $i\in J,$ we have by the lemma that $\Gd_\GG \simeq  r_i$ on $U_i$, hence
\begin{equation}
\label{5.4}
\|\dist(\nabla u,\SO(n))\|_{L^{p}(U_i)}\le\lambda^{\alpha}r_i^{-\alpha}
\|\Gd_\GG^{\alpha}\dist(\nabla u,\SO(n))\|_{L^{p}(U_i)}<\infty.
\end{equation}
The standard geometric rigidity estimate applied on $U_i$ yields $\BR_i\in \SO(n)$ with
$$
\|\nabla u-\BR_i\|_{L^{p}(U_i)}\le C\|\dist(\nabla u,\SO(n))\|_{L^{p}(U_i)},
$$
with
$C=C(n,p,L,R)$ independent of $i$. Multiplying by $(\lambda r_i)^{\alpha}$ and using the lemma we get
\begin{equation}
\label{5.5}
\|\Gd_\GG^{\alpha}(\nabla u-\BR_i)\|_{L^{p}(U_i)}\ \le C\|\Gd_\GG^{\alpha}\dist(\nabla u,\SO(n))\|_{L^{p}(U_i)}.
\end{equation}
Summing the $p$-th powers and using the finite overlap property of the cover we obtain
\begin{equation}
\label{5.6}
\sum_i\|\Gd_\GG^{\alpha}(\nabla u-\BR_i)\|^{p}_{L^{p}(U_i)}\ \le\ CE^{p}.
\end{equation}
Following [\ref{bib:Con.Zwi.}], we let  
$$
\BF=\sum_i\varphi_i\BR_i\in C^{\infty}(\GO;\mathbb R^{n\times n}),
$$
where the sum is locally finite by the lemma. Since $\sum_i\varphi_i=1$, thanks to the finite overlap property and by 
$\operatorname{supp}\varphi_i\cap\GO\subset U_i,$ we have by Jensen's inequality pointwise a.e. in $\GO:$
\begin{align*}
|\nabla u-\BF|&=\Big|\sum_i\varphi_i(\nabla u-\BR_i)\Big|\\
& \le\sum_i\varphi_i|\nabla u-\BR_i| \\
& \le C\Big(\sum_i\varphi_i|\nabla u-\BR_i|^{p}\Big)^{1/p},
\end{align*}
hence, by (\ref{5.5})-(\ref{5.6}) we obtain
\begin{align}
\label{5.7}
\|\Gd_\GG^{\alpha}(\nabla u-\BF)\|^{p}_{L^{p}(\GO)} & \le\sum_i\|\Gd_\GG^{\alpha}(\nabla u-\BR_i)\|^{p}_{L^{p}(U_i)}\\ \nonumber
 & \le CE^{p}.
\end{align}
Now since $\sum_i\nabla\varphi_i=0$ in $\GO$, we have
$$
\nabla \BF=\sum_i\nabla\varphi_i\otimes \BR_i=\sum_i\nabla\varphi_i\otimes(\BR_i-\nabla u).
$$
By the lemma  $\Gd_\GG|\nabla\varphi_i|\le\lambda^{2}$ on $U_i$ and $\nabla\varphi_i=0$ on $U_i^c$, so,
again by Jensen we have
\begin{align*}
\Gd_\GG^{1+\alpha}|\nabla \BF| & \le\lambda^{2}\sum_i\chi_{U_i}\Gd_\GG^{\alpha}|\nabla u-\BR_i|\\
&\le C\Big(\sum_j\chi_{U_j}\Gd_\GG^{\alpha p}|\nabla u-\BR_j|^{p}\Big)^{1/p},
\end{align*}
and therefore from (\ref{5.7}) we get
\begin{align}
\label{5.8}
\|\Gd_\GG^{1+\alpha}\nabla \BF\|_{L^{p}(\GO)} 
& \le C\Big(\sum_j\|\Gd_\GG^{\alpha}(\nabla u-\BR_j)\|^{p}_{L^{p}(U_j)}\Big)^{1/p}\\ \nonumber
& \le CE.
\end{align}
Now Theorem 2.4 (i) yields a constant matrix $\BR_*\in\mathbb R^{n\times n}$ with
\begin{align}
\label{5.9}
\|\Gd_\GG^{\alpha}(\BF-\BR_*)\|_{L^{p}(\GO)} & \le C\|\Gd_\GG^{1+\alpha}\nabla \BF\|_{L^{p}(\GO)}\\ \nonumber
& \le CE.
\end{align}
Combining (\ref{5.7}) and (\ref{5.9}) we arrive at $\ \|\Gd_\GG^{\alpha}(\nabla u-\BR_*)\|_{L^{p}(\GO)}\le CE$.
In the last step we replace $\BR_*$ by a proper rotation through a routine procedure. 
Let $\BR\in \SO(n)$ be a nearest point to $\BR_*$ in $\SO(n)$ (which is compact). For a.e. $x\in\GO$ we have
$$
|\BR_*-\BR|=\dist(\BR_*,\SO(n))\le|\BR_*-\nabla u(x)|+\dist(\nabla u(x),\SO(n)).
$$
Multiplying by $\Gd_\GG(x)^{\alpha}$ and taking $L^{p}(\GO)$-norms we get
$$
\|\Gd_\GG^{\alpha}(\BR_*-\BR)\|_{L^{p}(\GO)}\le\|\Gd_\GG^{\alpha}(\nabla u-\BR_*)\|_{L^{p}(\GO)}+E\le CE
$$
and one more triangle inequality gives (\ref{2.2}). This completes the proof of the theorem. 

\end{proof}


\section{Proof of Theorem~\ref{Thm:2.3}}
\setcounter{equation}{0}
\label{sec:plateWKorn}

In this section, all constants $c$ and $C$ may depend only on $S,c_1,c_2,\alpha, p$ and $n$ unless otherwise specified. Recall from the proof of Theorem~\ref{Thm:2.4} (ii), that there is $h_0>0$, depending only on $S$ and $c_1$, such that for $h\in(0,h_0)$ the map
$$
\Psi(x,t):=x+t\,n(x),\qquad x\in S,\ |t|\le c_1h,
$$
is injective and bi-Lipschitz with absolute constants, $\GO_h=\Psi\left(\{(x,t):-g_1^{h}(x)<t<g_2^{h}(x)\}\right)$ is an open set with 
$|\GO_h|\simeq h,$ $\diam\GO_h\simeq\diam S$, and
\begin{equation}
\label{6.1}
c \dist(x,\partial S)\le\Gd'(\Psi(x,t))\le C \dist(x,\partial S)\quad\text{for all}\quad x\in S, \  t\in (-g_1^{h}(x),g_2^{h}(x)).  
\end{equation}

We start with a uniform $h$-scale decomposition of $\Omega_h$ lemma.


\begin{lemma}[Uniform decomposition of $\Omega_h$]
\label{Lem:6.1}
Assume \textup{(H)} and \textup{(\ref{2.3})}. There exist constants $C_0,\tilde h_0,L_0,R_0>0,$ $M_0\ge1,$ and $N_0\in\mathbb N$, all depending only on $S,c_1,c_2$, such that for every $h\in(0,\tilde h_0)$ there are finitely many open sets $V_i^h\subset\GO_h,$ with the indexing set 
$I=\{1,2,\dots,K\}$ and a subset $I_\partial$ of $I$, functions $\varphi_i\in C^{\infty}(\GO_h,[0,1])$ for $i=1,2,\dots,K$ and for indices $i\in I_\partial$, there exist nonempty closed sets $\Gamma_i\subset\partial V_i^h$ such that the following conditions hols:
\begin{enumerate}
\item $\GO_h=\bigcup_{i\in I}V_i^h$ and each $V_i^h$ is a connected $(L_0,R_0)$-Lipschitz set with $\diam V_i^h\le M_0h$.
\item The sets $V_i^h$ have the uniformly bounded overlap property: $\sum_{i=1}^K\chi_{V_i^h}\le N_0.$
\item $\operatorname{supp}\varphi_i\cap\GO_h\subset V_i^h$, $\sum_{i=1}^K\varphi_i=1$ on $\GO_h$, and
$|\nabla\varphi_i|\le M_0/h$.
\item For $i \notin I_\partial$ we have $\ \max_{V_i^h}\Gd'\le C_0\min_{V_i^h}\Gd'$, while for 
 $i\in I_\partial$ we have  $\dist(x,\Gamma_i)\simeq \Gd'(x)$ for all $x\in V_i^h$.
\end{enumerate}
\end{lemma}


\begin{proof}[Proof of Lemma~\ref{Lem:6.1}]
Fix $M$ large, depending on $S,c_1,L_S$ (to be chosen later) and 
$\tilde h_0\le \min \left(\frac{h_0}{2},\frac{r_S}{2M(1+L_S)},\frac{r_S}{10M'}\right)$, where $h_0$ is the constant in the proof of part (ii) of Theorem~\ref{Thm:2.4}. We divide the proof into several steps.\\ 
\noindent\textbf{Step 1 (Construction).}
Let $\{x_i\}_{i\in I}\subset S$ be a maximal $h$-separated set. Due to the hypothesis (H), we have $K=|I|\simeq h^{-2}$. For $x\in S$
let $\pi_x$ be the projection operator onto the tangent space $T_xS.$ For $x\in S$ and $r\le r_S/10$ define further
\begin{align*}
N(x,r)&=\{z\in S:\ |\pi_x(z-x)|<r\},\\
V(x,r)&=\Psi\big(\{(z,t):z\in N(x,r),\ -g_1^{h}(z)<t<g_2^{h}(z)\}\big).
\end{align*}
Set initially for $h\in(0,\tilde h_0):$
$$
V_i^h=V(x_i,Mh)\quad\text{ and}\quad I_\partial=\Big\{i\in I \ :\ \inf_{V_i^h}\Gd'<M h\Big\}.
$$
For $i\in I_\partial$, the bounds in (\ref{6.1}) yield a point $x_0\in N(x_i,Mh)$ with $\dist(x_0,\partial S)\le CM h$,
hence a point $z_i\in\partial S$ with $|x_i-z_i|\le C'Mh$. We re-centre that patch by
replacing $V_i^h$ with $V(z_i,M' h)$, where $M'=(C'+2)M$, which contains the old $V_i^h$.
Thus we obtain finitely many patches of diameter $\le CM h$, centered at
$x_i\in S$ when $i\notin I_\partial$ or at $z_i\in\partial S$ when $i\in I_\partial$. The accompanying partition 
of unity in part 3 is obtained as in the last step of the proof of Lemma~\ref{Lem:5.1}. \\
\noindent\textbf{Step 2 (Proof of parts 1-4).} Now we prove that the constructed elements have the properties stated in the lemma. Let the point $w\in\mathbb R^3$ belong to $N$ sets $V_i^h.$ Then the ball $B_{Ch}(w)$ contains all the centers $x_i$ of that sets $V_i^h$ for some $C>0.$ Hence $B_{(C+1)h}(w)$ will contain all the corresponding balls $B_{h/2}(x_i),$ and since the balls $B_{h/2}(x_i),$ are disjoint, we get $|B_{(C+1)h}(w) |\geq \sum_i |B_{h/2}(x_i)|,$ which yields the bound $N\leq (2C+2)^3.$ This is part 2 of the lemma with 
$N_0=(2C+2)^3.$ The bound $\diam V_i^h\le M_0h$ and the inclusion $\GO_h=\bigcup_{i\in I}V_i^h$ with some
$M_0=CM$ and $N_0=N_0(S,M)$ in part 1 obviously follow from the construction and from the fact that the doubled 
balls $B_{2h}(x_i)$ cover $S$. Let's now verify the uniform Lipschitz property of $V_i^h.$ Rescale $V_i^h$ by $h^{-1}$ about its centre. In the chart of (H), surface $S$ becomes the graph of $y\mapsto h^{-1}\phi(hy)$, whose gradient is $\nabla\phi(hy)\to0$ uniformly as $h\to0$ by uniform continuity of $\nabla\phi$ by compactness. The rescaled thickness functions
$\hat a(y)=\frac{1}{h}g_1^{h}(x_i+hy)$ and $\hat b(y)=\frac{1}{h}g_2^{h}(x_i+hy)$ satisfy by (\ref{2.3}) the bounds
$$
1\le \hat a, \hat b\le c_1,\quad\text{and}\quad |\nabla \hat a|+|\nabla \hat b|=|\nabla g_1^{h}|+|\nabla g_2^{h}|\le c_2h \le c_2\tilde h_0,
$$
and for $i\in I_\partial$, the base domain is $\{|y|<M,\ y_2<h^{-1}\psi_{z_i}(hy_1)\}$ with the same Lipschitz constant $L_S$. Hence the rescaled sets are, for $h\le \tilde h_0$, uniformly $(L_0,R_0)$-Lipschitz and connected with for some $L_0$ and $R_0$ depending only on $S,c_1,c_2,$ and $M.$ The same is true for $V_i^h,$ because the $(R,L)$-Lipschitz property is scale-invariant. This completes parts 1 and 2 of the lemma and we turn to the next parts.

If $i\notin I_\partial$ then $\Gd'\ge Mh$ on $V_i^h$, while $\diam V_i^h\le M_0h$, hence 
$$
\max_{V_i^h}\Gd'\le\min_{V_i^h}\Gd'+M_0h\le \left(1+\frac{M_0}{M}\right)\min_{V_i^h}\Gd'=C_0\min_{V_i^h}\Gd',
$$
with $C_0=1+\frac{M_0}{M}.$ This is part 4 for interior patches. Let now $i\in I_\partial$, with centre $z_i\in\partial S$, and set
\begin{equation}
\label{6.2}
\Gamma_i=\overline{\Psi\big(\{(w,t)\ : \ w\in\partial S\cap N(z_i,M h),\ -g_1^{h}(x)\le t\le g_2^{h}(x)\}\big)}
\subset\partial V_i^h.
\end{equation}
Obviously $\Gamma_i\ne\emptyset$ and since $\Gamma_i\subset\partial_S\GO^h$, we have $\Gd' \le\dist(\cdot,\Gamma_i)$ on $V_i^h$. For the opposite bound, work in the chart of (H) at $z_i$ and use (\ref{6.1}). It suffices to show for $y=(y_1,y_2)$ in the base
domain with $|y|<M h$ and $y_2<\psi_{z_i}(y_1)$ that
\begin{equation}
\label{6.3}
\dist\big(y,\ \operatorname{graph}\psi_{z_i}|_{\{|s|<M h\}}\big)\ \le\ \sqrt{1+L_S^{2}}\ \dist(y,\partial S).
\end{equation}
Indeed, the vertical distance $v=\psi_{z_i}(y_1)-y_2$ bounds the left-hand side from above, because
the vertical projection $(y_1,\psi_{z_i}(y_1))$ belongs to the graph over $\{|s|<Mh\}$ as $|y_1|<Mh.$
Also, $v\le\sqrt{1+L_S^{2}}\,\dist(y,\operatorname{graph}\psi_{z_i})$ because $\operatorname{graph}\psi_{z_i}$ is $L_S$-Lipschitz. Finally $\dist(y,\operatorname{graph}\psi)=\dist(y,\partial S)$ in the chart by (H) and $\tilde h_0\le r_S/(2M)$,
all points of $\partial S$ at distance $\le r_S$ from $z_i$ lie in the chart, and
$$
\dist(y,\partial S)\le v\le(1+L_S)Mh\leq (1+L_S)M\tilde h_0\leq \frac{r_S}{2}.
$$
This proves (\ref{6.3}) and thus part 4 of the lemma for boundary patches.
\end{proof}


We prove (\ref{2.4}) and (\ref{2.5}) in the same way as (\ref{2.1}) and (\ref{2.2}), where we will employ part (ii) of Theorem~\ref{Thm:2.4} instead. 

\begin{proof}[Proof of Theorem~\ref{Thm:2.3}]

\textbf{The Ansatz-free bound.} Set $E=\|(\Gd')^{\alpha}\dist(\nabla u,\SO(3))\|_{L^{p}(\GO_h)}<\infty.$ The proof here follows exactly the lines in the proof of Theorem~\ref{Thm:2.2} with the only differences, that now due to the bound $|\nabla\varphi_i|\le M_0/h,$ 
an application of (\ref{2.6}) on each domain $v_i^h$ and the boundary portion $\GG_i,$ we will get the cllective estimate
\begin{equation*}
\|(\Gd')^{1+\alpha}\nabla \BF\|_{L^{p}(\GO^h)}\le\frac Ch
\Big(\sum_i\|(\Gd')^{\alpha}(\nabla u-\BR_i)\|^{p}_{L^{p}(V_i^h)}\Big)^{1/p}\le\frac ChE
\end{equation*}
instead (this is the only place where a factor $h^{-1}$ appears), and we will also need to apply the weighted Poincar\'e inequality in thin domains instead, i.e., part (ii) of Theroem~\ref{Thm:2.4} to the average field $\BF.$\\
\textbf{The Ansatz.} Assume that the midsurface $S$ contains a nonempty relatively open flat region. Thus,
after an isometry, there exists a bounded open set
$D\subset\mathbb{R}^2$ positive distance apart from the boundary of $S$ such that $D\times\{0\}\subset S.$ Choose a nonzero function $w\in C_c^\infty(D).$ For all sufficiently small $h>0$, the support of $w$ is separated by fixed positive distance from the lateral boundary of $\Omega_h.$ Consequently, for $w$ the distance function $\Gd'$ can be omitted from inequalities (\ref{2.4}) and (\ref{2.5}) and they both reduce to the corresponding non-weighted inequality. Both inequalities are well-understood for flat shells (plates) and it is known that a Kirchhoff--Love displacement 
$$
u^h(x_1,x_2,x_3)=
\begin{pmatrix}
-x_3\,\partial_1 w(x_1,x_2)\\[2mm]
-x_3\,\partial_2 w(x_1,x_2)\\[2mm]
w(x_1,x_2)
\end{pmatrix}
$$
in the local frame ($D$ is contained in the $x_1x_2-$plane) yields the desired scaling of Korn's constant in (\ref{2.4}), see [\ref{bib:Fri.Jam.Mue.1},\ref{bib:Fri.Jam.Mue.2}] or [Theorem~3.3, \ref{bib:Gra.Har.2}] for more details. By a well-known linearization procedure [\ref{bib:Fri.Jam.Mue.1},\ref{bib:Fri.Jam.Mue.2}], the Ansatz $x+h^2 u^h(x)$ will then work for (\ref{2.5}).

\end{proof}


\section*{Acknowledgements}
This material is supported by the National Science Foundation under Grants No. DMS-2206239.


\end{document}